\documentclass[12pt,reqno]{amsart}
\usepackage[colorlinks]{hyperref}
\usepackage{amssymb}
\usepackage{mathrsfs}
\usepackage[all,cmtip]{xy}
\usepackage{cite}
\usepackage{tikz}

\newtheorem{theorem}{Theorem}[section]
\newtheorem{corollary}[theorem]{Corollary}
\newtheorem{lemma}[theorem]{Lemma}
\newtheorem{proposition}[theorem]{Proposition}

\newtheorem{example}[theorem]{Example}
\theoremstyle{definition}

\numberwithin{equation}{section}
\newcommand{\Kspace}{\mathcal K}
\newcommand{\HK}{\mathcal H_K}
\newcommand{\D}{\mathbb D}

\newcommand{\Zp}{\mathbb Z_+}

\newcommand{\Ran}{\operatorname{Ran}}
\newcommand{\C}{\mathbb C}
\newcommand{\T}{\mathbf{T}}
\allowdisplaybreaks

\title[Weighted Frames]{Weighted Frames and the Douglas--Foias problem
}

\author{Hui Dan}
\address{Hui Dan: College of Mathematics, Sichuan University, Chengdu, 610065, China}
\email{hdan@scu.edu.cn}

\author{Kunyu Guo}
\address{Kunyu Guo: School of Mathematical Sciences, Fudan University, Shanghai, 200433, China}
\email{kyguo@fudan.edu.cn}

\author{Dilong Li}
\address{Dilong Li: School of Mathematical Sciences, Fudan University, Shanghai, 200433, China}
\email{dlli23@m.fudan.edu.cn}

\thanks{This work was supported by National Key R\&D Program of China (2024YFA1013400) and the National Natural Science Foundation of China (Grant Nos.~12231005 and 124B2005)}


\makeatletter
\@namedef{subjclassname@2020}{\textup{2020} Mathematics Subject Classification}
\makeatother

\begin{document}
%--------------------------------------Abstract----------------------------------------------
\begin{abstract}
Let $\mathbf{T}=(T_1,\ldots,T_n)$ be a commuting tuple of bounded
linear operators on a separable Hilbert space $\mathcal K$, and let
$\{a_\alpha\}_{\alpha\in\mathbb Z_+^n}$ be strictly positive weights.
Under mild assumptions on the weights, we characterize the weighted
frame vectors and weighted Parseval frame vectors, that is, those
vectors $\xi$ for which
$\{a_\alpha \mathbf{T}^\alpha\xi\}_{\alpha\in\mathbb Z_+^n}$
is a frame or a Parseval frame.
We reduce the problem to characterizing surjective module maps from
the reproducing analytic Hilbert module associated with the weights
onto the module $\mathcal K$, equipped with the module structure
defined by $\mathbf{T}$.
It is shown that weighted Parseval frame vectors differ by a unimodular
scalar whenever they exist.
If the kernel defined by the weights has the complete
Nevanlinna--Pick property, then any two associated weighted frame
vectors are similar whenever they exist.
We also focus on the case of the Hardy kernel on the polydisc,
that is, the case $a_\alpha=1$ for all $\alpha\in\mathbb{Z}^n_+$.
This is related to the famous Douglas--Foias problem: whether two
similar quotient modules of the Hardy spaces
$H^2(\mathbb D^n)~(n>1)$ are equal.
Our result shows that the answer is negative.
This leads to the conclusion that frame vectors for representations
of $\mathbb{Z}_+^n$ $(n>1)$ may fail to be similar.
It means that, for the case $a_\alpha=1$, there exists a tuple
$\mathbf{T}=(T_1,\dots,T_n)$ with $n>1$ on $\mathcal K$ and two
distinct frame vectors for $\mathbf{T}$ which are not similar.
\end{abstract}
\maketitle

%----------------------------------Text of article-------------------------------------------
\section{Introduction}

Let $\mathcal{K}$ be a separable Hilbert space, and let $\mathcal{I}$ be a countable index set. A family $\{ \xi_i\}_{i\in\mathcal{I}}\subseteq \mathcal{K}$ is said to be a \emph{frame} for $\mathcal{K}$ if there exist two constants $0<c\leq C<\infty$ such that
\begin{equation}\label{eq:frame-intro}
	c\|\eta\|^2\leq \sum_{i\in I}|\langle\eta,\xi_i\rangle|^2
	\leq C\|\eta\|^2,\qquad \eta\in\Kspace.
\end{equation}
If $c=C=1$, we say that the frame is \emph{Parseval}. The frame inequalities \eqref{eq:frame-intro} express the stability of recovery from the coefficients
$\{\langle\eta,\xi_i\rangle\}_{i\in I}$.
Frames generated by operator orbits arise naturally in dynamical sampling, where an unknown
state is recovered from observations of its evolution \cite{ADK13, ACMT17}.
Let $\mathbf{R}=(R_1,\ldots,R_n)$ be a commuting $n$-tuple of bounded
operators on a separable Hilbert space $\mathcal K$, and let
$\xi\in\mathcal K$ be a fixed observation vector.
At time $\alpha\in\mathbb Z_+^n$, where
$\mathbb Z_+=\{0,1,2,\ldots\}$, an unknown initial state
$\eta\in\mathcal K$ evolves into $\mathbf{R}^\alpha\eta$, where
$\mathbf{R}^\alpha=R_1^{\alpha_1}\cdots R_n^{\alpha_n}$.
Observing this state via $\xi$ gives the measurement
$\langle \mathbf{R}^\alpha\eta,\xi\rangle$.
The aim is to recover $\eta$ from these measurements.
By the identity
$$
\langle \mathbf{R}^\alpha\eta,\xi\rangle
=
\langle\eta,\T^\alpha\xi\rangle,
\qquad \alpha\in\mathbb Z_+^n,
$$
where $\T=(R_1^*,\ldots,R_n^*)$, the recovery problem leads to
the study of the family $\{T^\alpha\xi\}_{\alpha\in\mathbb Z_+^n}$.
The study of frames and Parseval frames of the form
$\{\T^\alpha\xi\}_{\alpha\in\mathbb Z_+^n}$ has recently attracted
considerable interest. We refer the reader to \cite{ACCP23,BC25,BHKLL26}
and the references therein.

In the single-operator setting, Christensen, Hasannasab, and Philipp~\cite{CHP20} characterized frames of the form
$\{T^m\xi\}_{m\geq 0}$ for a fixed operator $T$.
Their approach relies on the function theory of the one-variable Hardy space $H^2(\D)$, particularly Beurling's theorem.
They showed that $\{T^m\xi\}_{m\geq 0}$ is a frame if and only if
the pair $(T,\xi)$ is similar to the pair $(P_{N}M_z|_N, P_{N}1)$ for some model space $N$ of $H^2(\D)$.
Moreover, any two vectors generating such frames are related by an invertible operator commuting with $T$.

The corresponding questions for two commuting operators have also been studied in several settings. Aguilera, Cabrelli, Carbajal, and Paternostro~\cite{ACCP23} considered the bilateral and mixed bilateral--unilateral cases, indexed by $\mathbb Z^2$ and $\mathbb Z\times\Zp$, respectively. They studied frames with multiple generating vectors, using vector-valued
function spaces and shift operators.
For two commuting unilateral iterations, Bailey and Cabrelli~\cite{BC25}
characterized frames of the form
$\{T_1^iT_2^j\xi\}_{i,j\geq0}$ in terms of similarity to coordinate
multipliers compressed to quotient modules of $H^2(\D^2)$. We refer the readers to \cite{Yang01, Yang19} for related works on the Hardy module  $H^2(\D^2)$.

Bailey, Han, Kornelson, Larson, and Liu \cite{BHKLL26} studied
the similarity of frame vectors for semigroup representations.
This viewpoint is closely related to the theory of Hilbert modules over semigroups; see \cite{Guo02}.

In this paper, we consider weighted orbit families
\begin{equation}\label{eq:orbit}
	\{a_\alpha \T^\alpha\xi\}_{\alpha\in\Zp^n},
	\qquad
	\T^\alpha=T_1^{\alpha_1}\cdots T_n^{\alpha_n},
\end{equation}
where $\T=(T_1,\ldots,T_n)$ is a commuting $n$-tuple of bounded operators on $\Kspace$, $\xi\in\Kspace$, and the weights $\{a_\alpha\}_{\alpha\in\mathbb{Z}^n_+}$ are strictly positive.
After normalization, we may assume that $a_0=1$. A vector $\xi\in\Kspace$ is called a \emph{weighted (Parseval) frame vector} if the family in \eqref{eq:orbit} is a (Parseval) frame for $\Kspace$.
The questions are which vectors are weighted frame vectors or weighted Parseval frame vectors, and how these vectors are related.
To incorporate the weights, we introduce the formal rotation-invariant kernel
\begin{equation}\label{eq:kernel}
	K(z,\lambda)=\sum_{\alpha\in\Zp^n}
	a_\alpha^2 z^\alpha\overline{\lambda}^{\alpha}.
\end{equation}
Set $\Omega=\{z\in\C^n: K(z,z)<\infty\}$. Clearly, $0\in\Omega$. We further assume that $\Omega$ is open.
Then $\Omega$ is a complete Reinhardt domain.
Moreover, $K$ defines a kernel function on $\Omega$ (see \cite[Definition 2.22]{AM02}) that is analytic in $z$ and
anti-analytic in $\lambda$.
For each $\lambda\in\Omega$, we denote by $K_\lambda(\cdot)=K(\cdot,\lambda)$ the \emph{reproducing kernel function at $\lambda$}. Each $K_\lambda$ is analytic on $\Omega$.
On the linear span of $\{K_\lambda:\lambda\in\Omega\}$,
there is a canonical inner product $\langle\cdot,\cdot\rangle$ satisfying
$$
\langle K_\lambda,K_w\rangle=K(w,\lambda),
\qquad \lambda,w\in\Omega.
$$
The completion with respect to this inner product, denoted by $\HK$,
is called the \emph{reproducing kernel Hilbert space defined by $K$}; see~\cite{PR16} for more details.
Throughout the paper, we assume that the kernel $K$ satisfies
conditions \textup{(A1)--(A3)} below.
These assumptions ensure that the coordinate multiplication operators
$M_{z_j}$, $1\leq j\leq n$, are bounded on $\HK$, and
$\{a_\alpha z^\alpha\}_{\alpha\in\Zp^n}$ forms an orthonormal basis for $\HK$. Most classical analytic Hilbert spaces satisfy conditions
 \textup{(A1)--(A3)}, including the Hardy and Bergman spaces on the polydisc and the unit ball, and the Drury--Arveson space on the unit ball.

To state our result, we first recall a few basic notions from Hilbert module theory.
%Let $\T=(T_1,\dots,T_n)$ be a commuting tuple of bounded operators on $\Kspace$. Then
The tuple $\T$ induces a natural $\mathbb{C}[z_1,\dots,z_n]$-module structure on $\Kspace$ by
$$
p\cdot\eta=p(\T)\eta,\qquad p\in\mathbb{C}[z_1,\dots,z_n],\ \eta\in\Kspace.
$$
%Besides the module $\Kspace$, two other types of modules, $\HK$ and its quotient modules, will play an important role in this paper.
By our assumptions, the tuple $(M_{z_1},\dots,M_{z_n})$ induces a $\mathbb{C}[z_1,\dots,z_n]$-module structure on $\HK$.
A closed subspace of $\HK$ that is invariant under $M_{z_j}$ for $1\leq j\leq n$ naturally inherits the module structure, and is called a \emph{submodule}.
For a submodule $M$, let $N=M^\perp$. We define the
\emph{compressed coordinate multipliers} $\mathbf{S}=(S_1,\dots,S_n)$ on $N$ by
\begin{equation}
	S_j=P_NM_{z_j}|_N,\qquad 1\leq j\leq n.
\end{equation}
Since $M$ is a submodule,
\begin{equation}\label{eq:quo_mod_def}
	p(\mathbf{S})=S_p,\qquad p\in\mathbb{C}[z_1,\dots,z_n],
\end{equation}
where $S_p=P_NM_p|_N$.
Together with the module structure induced by
$\mathbf S$, the space $N$ is called a
\emph{quotient module} of $\HK$.
By \eqref{eq:quo_mod_def}, the module action of a polynomial $p$ is precisely given by the operator $S_p$.
Let $\T=(T_1,\dots,T_n)$ and $\T'=(T'_1,\dots,T'_n)$ be commuting
tuples of bounded operators on $\Kspace_1$ and $\Kspace_2$, respectively. 
With $\mathcal K_1$ and $\mathcal K_2$ equipped with the associated module structures,
an operator $A\in\mathcal B(\Kspace_1,\Kspace_2)$ is a module map if and only if $AT_j=T'_jA,~1\leq j\leq n$.
A module map is called a \emph{module isomorphism} if it is invertible.
The modules $\Kspace_1$ and $\Kspace_2$ are said to be \emph{similar} if there exists a module isomorphism between them, and \emph{unitarily equivalent} if there exists a unitary module isomorphism between them.

We now describe our results. Let $\Kspace$ be equipped with the module structure induced by $\T$. Then
Theorem~\ref{thm:xi_frame} shows that $\xi$ is a weighted frame vector if and only if there exists a surjective module map
$A:\HK\to\Kspace$ satisfying $A1=\xi$. Such a map is unique.
Moreover, the frame is Parseval if and only if $A$ is a coisometry, that is, $AA^*=I$.
This gives an existence criterion for weighted frame vectors in terms of the $\mathbb{C}[z_1,\dots,z_n]$-module structure on $\Kspace$.
More precisely, weighted frame vectors exist if and only if $\Kspace$ is similar to a quotient module of $\HK$, and a weighted Parseval frame vector exists if and only if $\Kspace$ is unitarily equivalent to a quotient module of $\HK$; see
Corollary~\ref{cor:xi}.

If weighted frame vectors exist, we say that two weighted frame vectors $\xi$ and $\eta$ are \emph{similar} if there exists a module automorphism $A$ of $\Kspace$ such that $A\xi=\eta$.
The classification of weighted frame vectors up to similarity is closely connected to the rigidity of quotient modules of analytic Hilbert spaces.
Using a lifting theorem for isometries, Douglas and Foias~\cite{DF93} established a rigidity theorem for unitarily equivalent quotient modules of the Hardy space $H^2(\mathbb{D}^n)$. Motivated by the theory of defect operators, we extend  their theorem to a broader class of analytic Hilbert spaces; see Proposition~\ref{prop:rig}.
Under the additional assumption that $K$ is a CNP kernel, a rigidity result holds for similar quotient modules; see Corollary~\ref{cor:rig}.
The proof uses the invariant subspace theorem of McCullough and
Trent~\cite{MT00}.
These rigidity results yield a classification of weighted frame vectors and weighted Parseval frame vectors.
We prove that any two weighted Parseval frame vectors differ by a unimodular scalar.
If $K$ is a CNP kernel, then all weighted frame vectors are similar; see Theorem~\ref{thm:cla}.

The rigidity of similar quotient modules fails for the Hardy space  $H^2(\D^n)$ with $n>1$.
Theorem~\ref{thm:Hardy} shows that there exist distinct quotient modules of $H^2(\D^n)~(n>1)$ that are similar.
This gives a negative answer to the Douglas--Foias problem \cite{DF93}.
We remark that this also yields a negative answer to an open
question raised in \cite{BHKLL26}, which asks whether any two
frame vectors for a fixed representation of a commutative
semigroup are necessarily similar; see Theorem~\ref{negaans}.

This paper is organized as follows.
In Section~2, we establish rigidity results for quotient modules of $\HK$ and  prove Theorem~\ref{thm:Hardy}, which gives a negative answer to the Douglas--Foias problem.
In Section~3, we apply these results to the existence and classification of weighted frame vectors and weighted Parseval frame vectors. For the case $a_\alpha=1$ for all
$\alpha\in\mathbb Z_+^n$ with $n>1$, we apply Theorem~\ref{thm:Hardy} to obtain a commuting $n$-tuple whose frame vectors need not be similar.

\section{Rigidity of quotient modules and the Douglas--Foias problem}

Let $n\geq 1$ and let $\{a_\alpha\}_{\alpha\in\mathbb{Z}^n_+}$ be strictly positive weights with $a_0=1$. We formally write
\begin{equation}\label{eq:formal_kernel}
	K(z,\lambda)=\sum_{\alpha\in\Zp^n}
	a_\alpha^2 z^\alpha\overline{\lambda}^{\alpha}.
\end{equation}
Set $\Omega=\{z\in\C^n: K(z,z)<\infty\}$. As mentioned in the introduction, we assume that $\Omega$ is open. Then $\Omega$ is a complete Reinhardt domain in $\C^n$.
Denote by $\HK$ the reproducing kernel Hilbert space defined by the kernel $K$. The space $\HK$ consists of all analytic functions
$f(z)=\sum_{\alpha\in\mathbb{Z}^n_+}
c_\alpha z^\alpha$ on $\Omega$ such that
$$
\|f\|^2
=\sum_{\alpha\in\mathbb{Z}^n_+}\frac{|c_\alpha|^2}{a_\alpha^2}<\infty.
$$
It is straightforward to check that $\{a_\alpha z^\alpha\}_{\alpha\in\Zp^n}$ forms an orthonormal basis for $\mathcal H_K$.
We impose the two conditions on the kernel $K$ and its associated reproducing kernel Hilbert space $\mathcal{H}_K$:
\begin{itemize}
	\item[(A1)] For each $z,\lambda\in\Omega$, $K(z,\lambda)\neq 0$;
	\item[(A2)] Each coordinate multiplication operator $M_{z_j}, 1\leq j\leq n,$ is bounded on $\HK$.
\end{itemize}
Note that $\Omega$ is bounded by (A2). Indeed, fix $\lambda\in\Omega$. By \textup{(A2)}, for every polynomial $p$,
the multiplication operator $M_p$ is bounded on $\HK$. It is straightforward to verify that
$$
M_p^*K_\lambda=\overline{p(\lambda)}K_\lambda,
$$
and hence $|p(\lambda)|\leq \|M_p\|$.
Taking $p(z)=z_j^m$ yields $|\lambda_j|\leq \|M_{z_j}^m\|^{1/m}$.
Letting $m\to\infty$ and using the spectral radius formula, we obtain
$$
|\lambda_j|\leq r_j,\qquad 1\leq j\leq n,
$$
where $r_j$ denotes the spectral radius of $M_{z_j}$.
Therefore, $\Omega\subset \D_{r_1}\times\cdots\times\D_{r_n},$
and hence $\Omega$ is bounded.

Moreover, \textup{(A1)} and \textup{(A2)} imply that $K$ is irreducible in the sense of Agler and McCarthy~\cite{AM02}; that is, for distinct $\lambda,w\in\Omega$, the reproducing kernel functions $K_\lambda$ and $K_w$
are linearly independent.
In addition, by \textup{(A1)}, we can write
\begin{equation}\label{eq:b_a}
	1-\frac{1}{K(z,\lambda)}=
	\sum_{\alpha\in\mathbb Z_+^n}
	b_\alpha z^\alpha\overline{\lambda^\alpha},
\end{equation}
where $b_0=0$ since $a_0=1$.
By the McCullough--Quiggin criterion~\cite{McC92,Qui93}, $K$ is a
complete Nevanlinna--Pick (CNP) kernel if and only if
$$
b_\alpha\geq 0
\qquad\text{for every }\alpha\in\mathbb Z_+^n.
$$

For most classical spaces, including the Hardy and Bergman spaces on the polydisc, as well as the
Hardy and Bergman spaces and the Drury--Arveson space on the unit ball, only finitely many coefficients $b_\alpha$ in
\eqref{eq:b_a} are nonzero.
 In this case, we have
\begin{equation}\label{eq:b_alpha}
	I-1\otimes 1=
	\sum_{\alpha\in\mathbb{Z}_+^n}
	b_\alpha M_{z^\alpha}M_{z^\alpha}^*,
\end{equation}
where, for $f,g\in\mathcal{H}_K$, the rank-one operator
$f\otimes g$ is defined by
\[
(f\otimes g)h=\langle h,g\rangle f,
\qquad h\in\mathcal{H}_K.
\]
Without the finite-support assumption on the coefficients $b_\alpha$,
we have the following characterization. In what follows, for $\alpha=(\alpha_1,\dots,\alpha_n),\beta=(\beta_1,\dots,\beta_n)\in\mathbb{Z}_+^n$, we set
$|\alpha|=\sum_{j=1}^n\alpha_j$ and write $\alpha\leq\beta$
if $\alpha_j\leq\beta_j$ for every $j=1,\dots,n$.
\begin{proposition}\label{prop:WOT}
	For each $N\geq 1$, let $A_N=	\sum_{|\alpha|\leq N}	b_\alpha M_{z^\alpha}M_{z^\alpha}^*.$
	Then the following statements are equivalent:
	\begin{itemize}
		\item[(i)] The sequence $\{A_N\}_{N=1}^\infty$ is uniformly
		bounded in operator norm.
		
		\item[(ii)] The sequence $\{A_N\}_{N=1}^\infty$ converges to
		$I-1\otimes 1$ in the strong operator topology.
		
		\item[(iii)] The sequence $\{A_N\}_{N=1}^\infty$ converges to
		$I-1\otimes 1$ in the weak operator topology.
		
		\item[(iv)] It holds
		\begin{equation}\label{eq:a_and_b}
				\sup_{N\geq 1}
			\sup_{\beta\in\mathbb{Z}_+^n}
			a^{-2}_\beta\left|	\sum_{|\alpha|\leq N, \alpha\leq \beta}
			b_\alpha a^2_{\beta-\alpha}
			\right|<\infty.
		\end{equation}
	\end{itemize}
\end{proposition}

\begin{proof}
	For each $\beta\in\mathbb{Z}_+^n$, let $e_\beta=a_\beta\,z^\beta$.
	Then $\{e_\beta:\beta\in\mathbb{Z}_+^n\}$ is an orthonormal basis
	for $\mathcal{H}_K$. Since $M_{z^\alpha}M_{z^\alpha}^*e_\beta$ equals
	$a_\beta^{-2}a_{\beta-\alpha}^2e_\beta$ when $\alpha\leq\beta$
	and is zero otherwise,
	we have $A_Ne_\beta=a_\beta^{-2}\bigl(\sum_{|\alpha|\leq N,\,\alpha\leq\beta}
	b_\alpha a_{\beta-\alpha}^2\bigr)e_\beta$, and hence
	\begin{equation*}
		\|A_N\|
		=
		\sup_{\beta\in\mathbb{Z}_+^n}
		a^{-2}_\beta\left|	\sum_{|\alpha|\leq N, \alpha\leq \beta}b_\alpha a^2_{\beta-\alpha}\right|.
	\end{equation*}
	This proves the equivalence of {\rm (i)} and {\rm (iv)}.
	
	Comparing coefficients in the identity
	\begin{equation}
		\left(1-\frac{1}{K(z,\lambda)}\right)K(z,\lambda)=K(z,\lambda)-1,
	\end{equation}
	we obtain
	\[
	\sum_{\alpha\leq\beta}
	b_\alpha a_{\beta-\alpha}^2
	=
	\begin{cases}
		a_\beta^2,
		&
		\beta\neq0,\\
		0,
		&
		\beta=0.
	\end{cases}
	\]
	Thus, for $N\geq|\beta|$, $A_Ne_\beta=(I-1\otimes1)e_\beta$. Then
	\begin{equation}\label{eq:AN}
		A_N p=(I-1\otimes 1)p,\qquad p\in\mathbb{C}[z_1,\dots,z_n]~\text{with}~ \deg p\leq N.
	\end{equation}
	If {\rm (i)} holds, then \eqref{eq:AN} and the
	density of the polynomials in $\mathcal H_K$ imply that $	A_N\rightarrow I-1\otimes1$
	in the strong operator topology. Therefore, {\rm (i)} implies {\rm (ii)}.
	
	The implication {\rm (ii)}$\Rightarrow${\rm (iii)} is immediate, while
	(iii)$\Rightarrow$(i) follows from the uniform boundedness
	principle.
\end{proof}

We further impose the following:
\begin{itemize}
	\item[(A3)] The weights $\{a_\alpha\}_{\alpha\in\mathbb Z_+^n}$
	satisfy one of the conditions in Proposition~\ref{prop:WOT}.
\end{itemize}
It is worth mentioning that an analogous condition in the strong
operator topology was imposed by Bhattacharyya and Jindal~\cite{BJ23} in their study of $1/k$-contractivity.

With these conditions in place, we now turn to the rigidity of quotient modules of $\HK$.
The following proposition extends the rigidity theorem of Douglas and
Foias~\cite{DF93} for quotient modules of $H^2(\D^n)$
to the present class of reproducing kernel Hilbert spaces. The proof of Douglas and Foias heavily relies on isometric lifting, while our argument uses
only the module structure of the reproducing kernel Hilbert space.

\begin{proposition}\label{prop:rig}
	Let $M_1$ and
	$M_2$ be two submodules of $\mathcal H_K$. Then the quotient modules $M_1^\perp$ and $M_2^\perp$ are unitarily equivalent if and only if $M_1=M_2$.
	Moreover, every unitary module automorphism of a quotient module is a unimodular scalar multiple of the identity map.
\end{proposition}

\begin{proof}
	It suffices to prove necessity. Let $U:M_1^\perp\rightarrow M_2^\perp$ be a unitary module isomorphism.
	Then for each polynomial $p$, $US_p=S_pU$,
	and hence $S^*_pU^*=U^* S^*_p$.
	Since $M_i^\perp,~i=1,2,$ are invariant under $M_p^*$, we have
	\[
	M_p^*U^*f=U^*M_p^*f,
	\qquad f\in M_2^\perp.
	\]
	In particular,
	\begin{equation}\label{eq:M_h}
		\langle M_pM_p^*U^*f,U^*f\rangle
		=
		\langle M_pM_p^*f,f\rangle,
		\qquad
		f\in M_2^\perp,\quad
		p\in\mathbb C[z_1,\dots,z_n].
	\end{equation}
	By {\rm (A3)},
	\[
	I-1\otimes1
	=
	\sum_{\alpha\in\mathbb Z_+^n}
	b_\alpha M_{z^\alpha}M_{z^\alpha}^*,
	\]
	where the series converges in the strong operator topology.
	Multiplying this identity on the left by $M_p$ and on the right by
	$M_p^*$, we obtain
	\[
	M_pM_p^*-p\otimes p
	=
	\sum_{\alpha\in\mathbb Z_+^n}
	b_\alpha M_{pz^\alpha}M_{pz^\alpha}^*.
	\]
	Applying \eqref{eq:M_h} to the polynomials $p$ and $pz^\alpha$
	yields
	\[
	\langle(p\otimes p)U^*f,U^*f\rangle
	=
	\langle(p\otimes p)f,f\rangle.
	\]
	Therefore,
	\[
	|\langle p,U^*f\rangle|
	=
	|\langle p,f\rangle|,
	\qquad
	p\in\mathbb C[z_1,\dots,z_n].
	\]
	Since the polynomials are dense in $\mathcal H_K$, it follows that
	\[
	|\langle h,U^*f\rangle|
	=
	|\langle h,f\rangle|,
	\qquad h\in\mathcal H_K.
	\]
	Consequently, for every nonzero $f\in M_2^\perp$, there exists a
	unimodular constant $\theta_f$ such that $U^*f=\theta_f f$.
	In particular, $M_2^\perp\subseteq M_1^\perp$. Applying the same argument to $U^*$, we obtain the reverse inclusion. Then $M_1^\perp=M_2^\perp$,
	and hence $M_1=M_2$.
	The converse is immediate.
	
	 Finally, when $M_1^\perp=M_2^\perp$,
	the linearity of $U^*$ shows that $\theta_f$ is independent of
	$f$. Thus $U^*=\theta I$
	for some $\theta\in\mathbb T$, and hence $U$ is a unimodular scalar
	multiple of the identity map.
\end{proof}

We now turn to similarity between quotient modules. Recall that
the multiplier algebra of $\mathcal H_K$ is defined by
$$
\operatorname{Mult}(\mathcal H_K)
:=
\left\{
h\in\operatorname{Hol}(\Omega):
hf\in\mathcal H_K
\text{ for every }f\in\mathcal H_K
\right\}.
$$
It is not difficult to check that polynomials are dense in ${\rm Mult}(\mathcal{H}_K)$ in the strong operator topology.

\begin{proposition}\label{prop:Mult_mod}
	Let $M_1$ and $M_2$ be submodules of $\mathcal{H}_K$, and suppose that $M_1^\perp$ is similar to $M_2^\perp$.
	Then
	\[
	\operatorname{Mult}(\mathcal{H}_K)\cap M_1
	=
	\operatorname{Mult}(\mathcal{H}_K)\cap M_2.
	\]
\end{proposition}

\begin{proof}
	Let $A:M_1^\perp\rightarrow M_2^\perp$
	be a module isomorphism. We set $B=AP_{M_1^\perp}$.
	Then
	\begin{equation}\label{eq:BMp}
		BM_p=S_p B,\qquad p\in\mathbb{C}[z_1,\dots,z_n],
	\end{equation}
 and hence
	\begin{equation}\label{eq:BMh}
		BM_h=S_h B,\qquad h\in {\rm Mult}(\mathcal{H}_K).
	\end{equation}
	Suppose now that $f\in\operatorname{Mult}(\mathcal{H}_K)\cap M_2$. Using a routine argument, one has $fB1\in M_2$. Then by \eqref{eq:BMh},
	$$
	Bf=BM_f 1=P_{M_2^\perp}\left(fB1 \right)=0.
	$$
	Since $A$ is injective, \(f\in M_1\). Thus $\operatorname{Mult}(\mathcal{H}_K)\cap M_2	\subseteq M_1$.
	Applying the same argument to the module isomorphism $A^{-1}$, we have $\operatorname{Mult}(\mathcal{H}_K)\cap M_1
	\subseteq M_2$.
	Then the desired equality follows.
\end{proof}

Suppose that $K$ is a CNP kernel. The invariant-subspace theorem of McCullough and
Trent \cite[Theorem 0.7]{MT00} implies that $\overline{\operatorname{Mult}(\mathcal H_K)\cap M}=M$
for every submodule $M$ of $\mathcal H_K$.
Together with Proposition~\ref{prop:Mult_mod}, this gives the following corollary.
\begin{corollary}\label{cor:rig}
	Suppose that $K$ is a CNP kernel. If $M_1$ and $M_2$ are submodules
	of $\mathcal H_K$, then the quotient modules $M_1^\perp$ and
	$M_2^\perp$ are similar if and only if $M_1=M_2$.
\end{corollary}

Classical examples of CNP kernels include the Hardy kernel and Dirichlet kernel on $\D$
and the Drury--Arveson kernel on the unit ball $\mathbb B_n$.
For a general kernel $K$ and the associated reproducing kernel Hilbert space $\mathcal{H}_K$, a natural question is whether two similar quotient modules of $\mathcal H_K$ must be equal.
For the Hardy space $H^2(\D^n)$ with $n>1$, this is the famous Douglas--Foias problem \cite{DF93, CG03}.
To attack this problem, we introduce some notions.
Let $\T=(T_1,\dots,T_n)$ be a commuting $n$-tuple of operators on
$\Kspace$, and equip $\Kspace$ with the associated module structure. The \emph{module point spectrum} $\sigma_p(\Kspace)$ consists of all points $\lambda=(\lambda_1,\dots,\lambda_n)\in\mathbb C^n$ such that
there exists a nonzero vector $\xi\in\Kspace$ satisfying
$$
T_j\xi=\lambda_j\xi,\qquad 1\leq j\leq n.
$$
Such a vector is called an \emph{eigenvector} corresponding to
$\lambda$.
 The module point spectrum $\sigma_p(\Kspace)$ is contained in the joint spectrum of the tuple $\T$. We refer the reader to~\cite{Tay70,Arv72, CG03} and the references therein for more background.

 Suppose that $M$ is a submodule of $\HK$ and $\lambda\in\sigma_p(M^\perp)\cap\Omega$. We can construct a quotient module
 similar to $M^\perp$ using an eigenvector corresponding to $\lambda$. More precisely,
 let $\varphi\in M^\perp$ be a nonzero eigenvector corresponding to $\lambda$, and define the rank-one operator
 \begin{equation}\label{eq:def_A}
 	A:\HK\to M^\perp,\qquad Af=f(\lambda)\varphi .
 \end{equation}
By choosing $\varphi$ with $\|\varphi\|<\|K_\lambda\|^{-1}$, we may  assume that $\|A\|<1$.
 For every $f\in\HK$ and $1\leq j\leq n$,
 \begin{equation*}
 	AM_{z_j}f
 	=\lambda_j f(\lambda)\varphi
 	=f(\lambda)S_j\varphi
 	=S_jAf.
 \end{equation*}
 Thus $A$ is a module map. We define
 \begin{equation}\label{eq:def_B}
 	B=P_{M^\perp}+A.
 \end{equation}
 Since both $P_{M^\perp}$ and $A$ are module maps, so is $B$.
 Furthermore,
 \begin{equation*}
 	B|_{M^\perp}
 	=I_{M^\perp}+A|_{M^\perp}
 \end{equation*}
 is invertible since $\|A\|<1$.
 It follows that $B$ is surjective.
Let $N=\ker B$. Then $N$ is a submodule of $\HK$, and 
a direct computation gives
\begin{equation}\label{eq:def_N}
	N=
	\left\{
	f-\frac{f(\lambda)}{1+\varphi(\lambda)}\varphi: 	f\in M
	\right\}.
\end{equation}
Moreover, $N^\perp$ is similar to $M^\perp$ via the module isomorphism $B|_{N^\perp}:N^\perp\to M^\perp$.

To continue this aim, for a subspace $E$ of $\HK$, the \emph{common zero set of $E$} is defined by
\begin{equation*}
	\mathcal{Z}(E) :=\{\lambda\in\Omega: f(\lambda)=0 \text{ for all }f\in E\}.
\end{equation*}
\begin{proposition}\label{prop:CZS_MandMult}
	If $M$ is a submodule of $\HK$
 such that $M\cap\operatorname{Mult}(\HK)\neq \{0\}$,
 then
	\[
	\sigma_p(M^\perp)\cap\Omega\subseteq \mathcal Z(M\cap\operatorname{Mult}(\HK)).
	\]
\end{proposition}
\begin{proof}
	We may assume that $\sigma_p(M^\perp)\cap\Omega\neq\emptyset$.
	Fix $\lambda\in\sigma_p(M^\perp)\cap\Omega$, and let
	$\varphi\in M^\perp$ be an eigenvector such that the corresponding module map $B$ defined by \eqref{eq:def_B} is surjective. Let $N=\ker B$.
	By the preceding argument, $N^\perp$ and $M^\perp$ are similar. For
	$g\in M\cap\operatorname{Mult}(\HK)$, by Proposition~\ref{prop:Mult_mod}
	we have $g\in N$. It follows from \eqref{eq:def_A} and \eqref{eq:def_B}	that $g(\lambda)=0$. Thus $\lambda\in\mathcal Z(M\cap\operatorname{Mult}(\HK))$.
\end{proof}

\begin{theorem}\label{thm:DF}
	Let $M$ be a submodule of $\HK$, and let
	$\lambda\in\sigma_p(M^\perp)\cap\Omega$ satisfy
	$\lambda\notin\mathcal Z(M)$.
	Suppose that $\varphi\in M^\perp$ is an eigenvector corresponding
	to $\lambda$ such that $\|\varphi\|<\|K_\lambda\|^{-1}$.
	Then
	$$
	N=\left\{f-\frac{f(\lambda)}{1+\varphi(\lambda)}\varphi:	f\in M	\right\}
	$$
	is a submodule of $\HK$ such that $N\neq M$, and the quotient modules
	$N^\perp$ and $M^\perp$ are similar.
\end{theorem}
\begin{proof}
	Let $A$ and $B$ be as constructed above. Then $N=\ker B$ is a submodule such that the quotient modules $N^\perp$ and $M^\perp$ are similar.
	 It remains to show that $N\neq M$.
	Since $\lambda\notin\mathcal Z(M)$, there exists a function
	$h\in M$ such that $h(\lambda)\neq0$. Then $Bh=h(\lambda)\varphi\neq0$,
	and hence $h\notin N$. It follows that $N\neq M$.
\end{proof}

We first construct an example for the Bergman space $L^2_a(\D)$ on the unit disc. Let $L$ be a submodule of $L^2_a(\D)$, and put
$S_z=P_{L^\perp}M_z|_{L^\perp}$. Then by \cite[Lemma~3.1]{ARS96},
for every $\lambda\in\D$,
the subspace $(z-\lambda)L$ is closed and
\begin{equation}\label{dimeq}
\dim\bigl(L\ominus(z-\lambda)L\bigr)
=\dim(L\ominus zL).
\end{equation}
Hedenmalm \cite{Hed93} proved that there exist a lot of submodules
$L$ of $L^2_a(\D)$ such that $\dim(L\ominus zL)\geq2$.
The following lemma shows that, for such a submodule $L$,
the module point spectrum of $L^\perp$ is equal to $\D$.
\begin{lemma}\label{lem:bergman-point}
If $\dim(L\ominus zL)\geq2$, then $\sigma_p(L^\perp)=\D$.
\end{lemma}
\begin{proof}
We fix $\lambda\in\D$.
By \eqref{dimeq}, there exists a nonzero function $h\in L\ominus(z-\lambda)L$ with
$h(\lambda)=0$. Set $g(z)=h(z)/(z-\lambda)$.
Then $g\in L^2_a(\D)$. 
If $g\in L$, then $h=(z-\lambda)g\in(z-\lambda)L$, 
which contradicts $h\perp(z-\lambda)L$. Therefore, we have
$u=P_{L^\perp}g\neq0$. Since $L$ is a submodule,
$$
(z-\lambda)u=h-(z-\lambda)P_Lg\in L,
$$
and hence $S_zu=\lambda u$. This proves
$\D\subseteq\sigma_p(L^\perp)$.

For the reverse inclusion, let $\lambda\in\sigma_p(L^\perp)$
and choose a nonzero $f\in L^\perp$ such that $S_zf=\lambda f$.
Then
$$
|\lambda|^2\|f\|^2=\|S_zf\|^2
\leq\|zf\|^2=\frac{1}{\pi}\int_{\D}|z|^2|f(z)|^2\,dA(z)<\|f\|^2.
$$
Thus $|\lambda|<1$, and hence $\sigma_p(L^\perp)\subset\D$.
Combining this with the opposite inclusion yields
$\sigma_p(L^\perp)=\D$.
\end{proof}

\begin{example}\label{ex:Bergman}
Using Hedenmalm's theorem \cite{Hed93}, we can take a submodule $L$ of $L^2_a(\D)$ such that $\dim(L\ominus zL)\geq 2$.
Choose $\lambda\in\D\setminus \mathcal{Z}(L)$. 
Then Lemma~\ref{lem:bergman-point} implies that $\lambda\in \sigma_p(L^\perp)$. Therefore, by Theorem~\ref{thm:DF}, there exists a submodule $N\neq L$ such that $N^\perp$ and $L^\perp$ are similar quotient modules of $L^2_a(\D)$.
\end{example}

We now turn to the Hardy space $H^2(\D^n)$ on the polydisc $\D^n$ with $n>1$. It is straightforward to verify that, for every submodule
$M$ of $H^2(\D^n)$, $\sigma_p(M^\perp)\subseteq \D^n$.
To transfer the construction from the Bergman space to this setting, we first consider the case $n=2$. Recall that the diagonal restriction operator defined by
Rudin \cite{Rudin},
$$
\Phi:H^2(\D^2)\to L^2_a(\D), \qquad (\Phi f)(z)=f(z,z),
$$
is a coisometry; i.e., $\Phi\Phi^*=I$.

\begin{proposition}\label{prop:diagonal-spectrum}
	Let $L$ be a submodule of $L^2_a(\D)$ with  $\dim(L\ominus zL)\geq2$,  and put $M=\Phi^{-1}(L)$.
	Then $M$ is a submodule of $H^2(\D^2)$, and
	\begin{align}
		\sigma_p(M^\perp)
		&=\{(\lambda,\lambda):\lambda\in\D\},
		\label{eq:diagonal-point-spectrum}\\
		\mathcal Z(M)
		&=\{(\lambda,\lambda):\lambda\in\mathcal Z(L)\}.
		\label{eq:diagonal-zero-set}
	\end{align}
\end{proposition}

\begin{proof}
	Since  $\Phi M_{z_j}=M_z\Phi$ for $j=1,2$,
	$M$ is a submodule of $H^2(\D^2)$.
	As $\Phi\Phi^*=I$, we have the
	orthogonal decompositions
	$$
	H^2(\D^2)=\ker\Phi\oplus\Phi^*L^2_a(\D),\qquad
	M=\ker\Phi\oplus\Phi^*L.
	$$
	It follows that
	$$
	M^\perp=\Phi^*(L^\perp),\qquad
	P_{M^\perp}=\Phi^*P_{L^\perp}\Phi.
	$$
	Thus $U=\Phi^*|_{L^\perp}:L^\perp\to M^\perp$ is unitary. For $u\in L^\perp$ and $j=1,2$,
	\[
	S_jUu
	=\Phi^*P_{L^\perp}\Phi M_{z_j}\Phi^*u
	=\Phi^*P_{L^\perp}M_zu
	=US_zu,
	\]
	and hence \eqref{eq:diagonal-point-spectrum} follows from Lemma~\ref{lem:bergman-point}.
	
	Now we show \eqref{eq:diagonal-zero-set}. Since $z_1-z_2\in\ker\Phi\subseteq M$,  every common zero of $M$ lies on the diagonal. Furthermore,
	$\Phi(M)=L$ implies that for $\lambda\in\D$,
	$
	F(\lambda,\lambda)=0$ for every $F\in M$ if and only if $h(\lambda)=0$ for every $h\in L$.
	This proves \eqref{eq:diagonal-zero-set}.
\end{proof}

\begin{theorem}\label{thm:Hardy}
	For every $n>1$, there exists a submodule $M$ of $H^2(\D^n)$ such that $\sigma_p(M^\perp)\setminus\mathcal Z(M)\neq\emptyset$.
	Moreover, for such a submodule $M$, a submodule $N$ constructed as in
	Theorem~\ref{thm:DF} satisfies $N\neq M$ and $N^\perp$ is similar
	to $M^\perp$.
\end{theorem}
\begin{proof}
	Let $L$ be the Bergman submodule in Example~\ref{ex:Bergman}, and
	choose $\lambda\in\D\setminus\mathcal Z(L)$.
	For $n=2$, set $M_2=\Phi^{-1}(L)$.
	Proposition~\ref{prop:diagonal-spectrum} gives $(\lambda,\lambda)\in
	\sigma_p(M_2^\perp)\setminus\mathcal Z(M_2)$.
	This proves the first assertion for $n=2$.
	
	For $n>2$, define $R:H^2(\D^n)\to H^2(\D^2)$ by putting
	\[
	(RF)(z_1,z_2)=F(z_1,z_2,0,\ldots,0),
	\]
	and set $M_n=R^{-1}(M_2)$. The operator $R$ is a coisometry, with
	$R^*f=f\otimes1$ under the identification
	\[
	H^2(\D^n)=H^2(\D^2)\otimes H^2(\D^{n-2}).
	\]
	Clearly, $M_n$ is closed and invariant under all coordinate
	multipliers, and hence is a submodule of $H^2(\D^n)$.
	Moreover, $M_n^\perp=M_2^\perp\otimes\C1$.
	It is not difficult to check that for every $f\in M_2^\perp$,
	$S_j(f\otimes1)=(S_jf)\otimes1$ for $j=1,2$, while
	$S_j(f\otimes1)=0$ for $j>2$.
	Consequently, if $\phi\in M_2^\perp$ is an eigenvector corresponding to $(\lambda,\lambda)$, then $\phi\otimes1$
	is an  eigenvector corresponding to $\lambda'=(\lambda,\lambda,0,\ldots,0)\in\D^n$.
	Since $(\lambda,\lambda)\notin\mathcal Z(M_2)$, there exists
	$F\in M_2$ with $F(\lambda,\lambda)\neq0$. The function
	$R^*F=F\otimes1$ belongs to $M_n$ and $R^*F(\lambda')=F(\lambda,\lambda)\neq0$.
	Thus $\lambda'\notin\mathcal Z(M_n)$.
	
	In either case, applying Theorem~\ref{thm:DF} to $M=M_n$ ,
	%and the eigenvalue constructed above, 
	one can choose a submodule $N\neq M$ with $N^\perp$ similar to $M^\perp$.
\end{proof}

We remark that Theorem~\ref{thm:Hardy} shows that the Douglas--Foias
problem has a negative answer.
For more background on analytic Hilbert spaces and quotient modules, we refer the reader to \cite{CG03, AM02, GRS02, GW09, Guo99} and the references therein.

\section{Similarity of weighted frame vectors}
In Section~2, we developed rigidity results for quotient modules of
analytic Hilbert spaces. In this section, we apply these results to
the study of weighted frame vectors.
Recall that $\T=(T_1,\dots,T_n)$ is a commuting $n$-tuple of operators
on a separable Hilbert space $\mathcal K$. As in Section~2, equip
$\Kspace$ with the module structure induced by $\T$. Let
$\{a_\alpha\}_{\alpha\in\mathbb Z_+^n}$ be strictly positive weights.
For a weighted frame vector $\xi$, the associated \emph{analysis operator}
$\Theta_\xi:\mathcal K\to\mathcal H_K$ is defined by
\begin{equation}
	\Theta_\xi\eta
	=
	\sum_{\alpha\in\mathbb Z_+^n}
	\langle\eta,a_\alpha \T^\alpha\xi\rangle
	a_\alpha z^\alpha, \qquad \eta\in\Kspace.
\end{equation}
The adjoint operator $\Theta_\xi^*:\HK\to\Kspace$ is called the \emph{synthesis operator}.
Since $\{a_\alpha z^\alpha\}_{\alpha\in\mathbb{Z}^n_+}$ forms an
orthonormal basis for $\HK$, we have
$$
\|\Theta_\xi\eta\|^2
=\sum_{\alpha\in\mathbb Z_+^n}
\left|\langle\eta,a_\alpha \T^\alpha\xi\rangle\right|^2,
\qquad
\eta\in\mathcal K .
$$
By the definition of weighted frame vectors, there
exist constants $c_\xi,~C_\xi>0$ such that
\begin{equation}\label{eq:Theta_bdd}
	c_\xi\|\eta\|^2
	\leq
	\|\Theta_\xi\eta\|^2
	\leq
C_\xi\|\eta\|^2,
	\qquad
	\eta\in\mathcal K .
\end{equation}
Clearly, the synthesis operator $\Theta_\xi^*$
is surjective, and $\xi$ is a weighted Parseval frame vector if and only if
$\Theta_\xi$ is an isometry.

\begin{theorem}\label{thm:quo_mod_map}
Let $\xi$ be a weighted frame vector for the tuple $\T$. Then
	\begin{equation}\label{eq:Theta_poly}
		M_p^*\Theta_\xi
		=
		\Theta_\xi p(\T)^*,
		\qquad
		p\in\mathbb C[z_1,\dots,z_n].
	\end{equation}
	Consequently, $\operatorname{Ran}\Theta_\xi$ is a quotient module of $\mathcal H_K$, and the synthesis operator $\Theta_\xi^*:\mathcal H_K\rightarrow\mathcal K$
	is a surjective module map satisfying $\Theta_\xi^*1=\xi$.
	Moreover, the restriction $\Theta_\xi^*\big|_{\operatorname{Ran}\Theta_\xi}
	:
	\operatorname{Ran}\Theta_\xi\rightarrow\mathcal K$
	is a module isomorphism.
	In particular, if $\xi$ is a weighted Parseval frame vector, then
	$\Theta_\xi^*|_{\operatorname{Ran}\Theta_\xi}$ is unitary.
\end{theorem}
\begin{proof}
	We first show that
	\begin{equation}\label{eq:Theta_zj}
		M_{z_j}^*\Theta_\xi
		=
		\Theta_\xi T_j^*,
		\qquad 1\leq j\leq n.
	\end{equation}
	Fix $1\leq j\leq n$ and $\eta\in\mathcal K$, and set
	$e_j=(0,\dots,0,1,0,\dots,0)$, with $1$ in the $j$-th position.
	 Since $M_{z_j}^*(a_\alpha^2z^\alpha)$ equals $a_{\alpha-e_j}^2z^{\alpha-e_j}$ when $\alpha\geq e_j$ and is zero otherwise,
	we obtain
	\begin{align*}
		M_{z_j}^*\Theta_\xi\eta
		&=
		\sum_{\alpha\geq e_j}
		\langle\eta,\T^\alpha\xi\rangle
		a_{\alpha-e_j}^2z^{\alpha-e_j}\\
		&=
		\sum_{\beta\in\mathbb Z_+^n}
		\langle\eta,\T^{\beta+e_j}\xi\rangle
		a_\beta^2z^\beta\\
		&=
		\sum_{\beta\in\mathbb Z_+^n}
		\langle T_j^*\eta,\T^\beta\xi\rangle
		a_\beta^2z^\beta\\
		&=
		\Theta_\xi T_j^*\eta.
	\end{align*}
	This proves \eqref{eq:Theta_zj}, and \eqref{eq:Theta_poly} follows by
	linearity and iteration.
	Since $\operatorname{Ran}\Theta_\xi$ is closed,
	\eqref{eq:Theta_zj} shows that it is a quotient module of
	$\mathcal H_K$.
	Taking adjoints in \eqref{eq:Theta_poly}, we obtain
	\[
	\Theta_\xi^*M_p
	=
	p(\T)\Theta_\xi^*,
	\qquad
	p\in\mathbb C[z_1,\dots,z_n].
	\]
	Thus $\Theta_\xi^*:\mathcal H_K\rightarrow\mathcal K$ is a module map. Since $\Theta^*_\xi$ is surjective,
	\[
	\Theta_\xi^*\big|_{\operatorname{Ran}\Theta_\xi}
	:
	\operatorname{Ran}\Theta_\xi\rightarrow\mathcal K
	\]
	is a module isomorphism.
	It remains to verify $\Theta_\xi^*1=\xi.$
	Indeed, using $a_0=1$ and the
	orthogonality of the monomials, we have for every $\eta\in\mathcal K$,
	$$
	\langle\Theta_\xi^*1,\eta\rangle
	=
	\langle1,\Theta_\xi\eta\rangle\\
	=
	\left\langle
	1,
	\sum_{\alpha\in\mathbb Z_+^n}
	\langle\eta,\T^\alpha\xi\rangle
	a_\alpha^2z^\alpha
	\right\rangle\\
	=
	\langle\xi,\eta\rangle.
	$$
	It follows that $\Theta_\xi^*1=\xi$.
	
	Finally,  if $\xi$ is a weighted Parseval frame vector, then $\Theta_\xi$ is an
	isometry. Therefore, $\Theta_\xi^*\big|_{\operatorname{Ran}\Theta_\xi}$
	is unitary.
\end{proof}

The preceding theorem yields a von Neumann type inequality for the
$\operatorname{Mult}(\HK)$ functional calculus of the tuple $\T$.

\begin{corollary}
	Let $\xi$ be a weighted frame vector for the tuple $\T$. Then there
	exists a constant $\kappa_\xi\geq1$ such that
	\begin{equation}\label{eq:p(T)_kappa}
		\|h(\T)\|\leq\kappa_\xi	\|h\|_{\operatorname{Mult}(\mathcal H_K)},
		\qquad
		h\in\operatorname{Mult}(\HK).
	\end{equation}
	More precisely, if $c_\xi$ and $C_\xi$ are lower and upper frame
	bounds in \eqref{eq:Theta_bdd}, respectively, then
	$\kappa_\xi$ can be chosen so that $1\leq \kappa_\xi\leq \sqrt{C_\xi/c_\xi}$.
	In particular,  if $\xi$ is a weighted Parseval frame vector, then
	\begin{equation}\label{eq:p(T)}
		\|h(\T)\|\leq\|h\|_{\operatorname{Mult}(\mathcal H_K)},
		\qquad
		h\in\operatorname{Mult}(\HK).
	\end{equation}
\end{corollary}

\begin{proof}
	We set $X= \Theta_\xi^*\big|_{\operatorname{Ran}\Theta_\xi}$.
	By Theorem~\ref{thm:quo_mod_map}, the map $X:\operatorname{Ran}\Theta_\xi\to\mathcal K$ is a module
	isomorphism. Then
	\[
	XS_p=p(\T)X,\qquad p\in\mathbb{C}[z_1,\dots,z_n],
	\]
	and hence $p(\T)=XS_pX^{-1}$.
	It follows that
	$$
	\|p(\T)\|\leq\|X\|\,\|X^{-1}\|\,\|S_p\|
	\leq	\|X\|\,\|X^{-1}\|	\|p\|_{\operatorname{Mult}(\mathcal H_K)}.
	$$
	Since polynomials are dense in $\operatorname{Mult}(\HK)$ in the
	strong operator topology, \eqref{eq:p(T)_kappa} holds with the constant
	\begin{equation}\label{eq:kappa_xi}
		\kappa_\xi=\|X\|\,\|X^{-1}\|\geq 1.
	\end{equation}
	Moreover, the frame inequalities imply that
	\[
	c_\xi I
	\leq
	\Theta_\xi^*\Theta_\xi
	\leq
	C_\xi I.
	\]
	We obtain $\|X\|\leq\sqrt{C_\xi}$ and $\|X^{-1}\|\leq 1/\sqrt{c_\xi}$.
	Combining these estimates with \eqref{eq:kappa_xi} yields that $\kappa_\xi\leq
	\sqrt{C_\xi/ c_\xi} $.
	
If $\xi$ is a weighted Parseval frame vector, then the corresponding constants
$c_\xi$ and $C_\xi$ may be taken to be $1$, and hence \eqref{eq:p(T)} follows.
\end{proof}

We can now characterize weighted frame vectors.
\begin{theorem}\label{thm:xi_frame}
	Let $\xi\in\Kspace$.
	Then $\xi$ is a weighted frame vector for the tuple $\T$ if and only if there exists a
	surjective module map $A:\mathcal H_K\to\mathcal K$
	such that $A1=\xi$.
	Such a map $A$ is necessarily unique.
	Moreover, $\xi$ is a weighted Parseval frame vector for the tuple $\T$ if and only if $A$ is a
	coisometry, that is, $AA^*=I$.
\end{theorem}
\begin{proof}
	It is not difficult to see that if such a module map $A$ exists, then it is unique.
	Suppose first that $\xi$ is a weighted frame vector. By
	Theorem~\ref{thm:quo_mod_map}, $\Theta_\xi^*:\mathcal H_K\rightarrow\mathcal K$
	is a surjective module map satisfying $\Theta_\xi^*1=\xi$. Thus one
	may take $A=\Theta_\xi^*$. If the frame is Parseval, then $A^*=\Theta_\xi$
	is an isometry and hence $AA^*=I$.
	
	Conversely, suppose that there exists a surjective module map $	A:\mathcal H_K\rightarrow\mathcal K$ with
	$A1=\xi$. Note that $\{a_\alpha z^\alpha\}_{\alpha\in\mathbb{Z}^n_+}$
	forms an orthonormal basis for $\mathcal H_K$ and
	$$
	a_\alpha Az^\alpha=a_\alpha AM_{z^\alpha}1=a_\alpha \T^\alpha \xi,\qquad \alpha\in\mathbb{Z}^n_+.
	$$
	Then for every $\eta\in\mathcal{K}$, we have
	\begin{align*}
		\sum_{\alpha\in\mathbb Z_+^n}
		\left|
		\langle\eta,a_\alpha \T^\alpha\xi\rangle
		\right|^2
		&=
		\sum_{\alpha\in\mathbb Z_+^n}
		\left|
		\langle\eta, a_\alpha Az^\alpha\rangle
		\right|^2\\
		&=
		\sum_{\alpha\in\mathbb Z_+^n}
		\left|
		\langle A^*\eta, a_\alpha z^\alpha\rangle
		\right|^2\\
		&=
		\|A^*\eta\|^2.
	\end{align*}
	Since $A$ is surjective, $A^*$ is bounded below. Hence there exists
	$c>0$ such that
	$$
	c\|\eta\|	\leq	\|A^*\eta\|
	\leq	\|A\|\|\eta\|,
	\qquad \eta\in\mathcal K.
	$$
	Therefore,
	$$
	c^2\|\eta\|^2	\leq
	\sum_{\alpha\in\mathbb Z_+^n}
	\left|	\langle\eta,a_\alpha \T^\alpha\xi\rangle	\right|^2
	\leq	\|A\|^2\|\eta\|^2,\qquad \eta\in\mathcal{K}.
	$$
	It follows that $\{a_\alpha \T^\alpha\xi\}_{\alpha\in\mathbb Z_+^n}$
	is a frame for $\mathcal K$, and hence $\xi$ is a weighted frame vector.

	Finally, the identity
	$$
	\sum_{\alpha\in\mathbb Z_+^n}
	\left|
	\langle\eta,a_\alpha \T^\alpha\xi\rangle
	\right|^2
	=
	\|A^*\eta\|^2
	$$
	shows that $\xi$ is a weighted Parseval frame vector if and only if $AA^*=I$.
\end{proof}

Theorem~\ref{thm:xi_frame} yields the following existence criteria for weighted frame vectors.
\begin{corollary}\label{cor:xi}
There exists a weighted frame vector for the tuple $\T$ if and only if $\mathcal K$ is
	similar to a quotient module of $\mathcal H_K$.
	Moreover, there exists a weighted Parseval frame vector for the tuple $\T$  if and only if
	$\mathcal K$ is unitarily equivalent to a quotient module of
	$\mathcal H_K$.
\end{corollary}

\begin{proof}
	Necessity in both assertions follows immediately from
	Theorem~\ref{thm:xi_frame}, so it remains to prove sufficiency.
	
	Suppose that $B:N\rightarrow\mathcal K$
	is a module isomorphism, where $N=M^\perp$ is a quotient module of
	$\mathcal H_K$. Then $A:=BP_N:\mathcal H_K\rightarrow\mathcal K$
	is a surjective module map. By Theorem~\ref{thm:xi_frame}, $A1$ is a weighted frame vector. Moreover, if $B$ is unitary, then $AA^*=I$, and hence $A1$ is a weighted Parseval frame vector.
\end{proof}

Recall that for a weighted frame vector $\xi$, the analysis operator
$\Theta_\xi$ maps $\mathcal K$ into $\mathcal H_K$.
The following proposition shows that the similarity of weighted frame
vectors is determined by the ranges of their analysis operators.
\begin{proposition}\label{prop:equi_frame}
	Let $\xi$ and $\eta$ be two weighted frame vectors for the tuple $\T$.
	Then $\xi$ and $\eta$ are similar if and only if $\operatorname{Ran}\Theta_\xi
	=
	\operatorname{Ran}\Theta_{\eta}$.
\end{proposition}

\begin{proof}
	Suppose first that $\xi$ and $\eta$ are similar, and let $A$
	be a module automorphism on $\Kspace$ satisfying
	$A\xi=\eta$. For every $\alpha\in\mathbb Z_+^n$, we have
	\[
	A\Theta_\xi^*z^\alpha
	=
	A\T^\alpha\xi
	=
	\T^\alpha\eta
	=
	\Theta_{\eta}^*z^\alpha.
	\]
	Since the polynomials are dense in $\mathcal H_K$,  $A\Theta_\xi^*=\Theta_\eta^*$. Thus $\ker\Theta_\xi^*	=	\ker\Theta_\eta^*$,
	and hence $\operatorname{Ran}\Theta_\xi=
	\operatorname{Ran}\Theta_\eta$.
	
	Conversely, suppose that $	N:=\operatorname{Ran}\Theta_\xi	=	\operatorname{Ran}\Theta_\eta.$
	Then
	\begin{equation}\label{eq:equi_frame_A}
		A:=
		\Theta_\eta^*\big|_N
		\circ
		\left(\Theta_\xi^*\big|_N\right)^{-1}
		:
		\mathcal K\rightarrow\Kspace
	\end{equation}
	is a module automorphism. Moreover,
	\[
	A\xi
	=
	\Theta_\eta^*P_N1
	=
	\Theta_\eta^*1
	=\eta.
	\]
	Therefore, $\xi$ and $\eta$ are similar.
\end{proof}

We now classify weighted frame vectors.
The weights $\{a_\alpha\}_{\alpha\in\mathbb Z_+^n}$ are called
\emph{CNP weights} if the associated kernel $K$ is a CNP kernel.
\begin{theorem}\label{thm:cla}
	The following statements hold.
	\begin{itemize}
		\item[(i)] If $\xi$ and $\eta$ are two weighted Parseval frame
		vectors for the tuple $\T$, then there exists a unimodular
		constant $c$ such that $\eta=c\xi$.
		
		\item[(ii)] If $\{a_\alpha\}_{\alpha\in\mathbb Z_+^n}$ are CNP
		weights and $\xi,~\eta$ are two weighted frame vectors for
		the tuple $\T$, then $\xi$ and $\eta$ are similar.
	\end{itemize}
\end{theorem}

\begin{proof}
	We first prove {\rm (i)}.
	By Theorem~\ref{thm:quo_mod_map}, the operators
	$$
	U_\xi	=	\Theta_{\xi}^*	\big|_{\operatorname{Ran}\Theta_{\xi}}
	\quad\text{and}\quad
	U_\eta	=	\Theta_\eta^*	\big|_{\operatorname{Ran}\Theta_\eta}
	$$
	are unitary module isomorphisms onto $\mathcal K$. Hence $V	:=	U_\xi^{-1}U_\eta	:
	\operatorname{Ran}\Theta_\eta
	\to
	\operatorname{Ran}\Theta_{\xi}$
	is a unitary module isomorphism. By Proposition~\ref{prop:rig},
	\[
	\operatorname{Ran}\Theta_\eta
	=
	\operatorname{Ran}\Theta_\xi
	=:
	N.
	\]
	Moreover, $V=cI_N$ for some unimodular constant $c$. Therefore,
	$$
	\eta=U_\eta(P_N1)=U_\xi V(P_N1)=cU_\xi(P_N1)=c\xi.
	$$
	
	For {\rm (ii)},
	by Theorem~\ref{thm:quo_mod_map}, $\Kspace$ is similar to both
	$\operatorname{Ran}\Theta_\xi$ and $\operatorname{Ran}\Theta_\eta$.
	It follows that $\operatorname{Ran}\Theta_\xi$ and
	$\operatorname{Ran}\Theta_\eta$ are similar quotient modules of $\HK$.
	Applying Corollary~\ref{cor:rig} and the CNP assumption, we have
	$
	\operatorname{Ran}\Theta_\xi
	=
	\operatorname{Ran}\Theta_\eta.
	$
	It then follows from Proposition~\ref{prop:equi_frame} that
	$\xi$ and $\eta$ are similar.
\end{proof}

In contrast to the CNP setting of Theorem~\ref{thm:cla}{\rm (ii)}, we
conclude the paper by considering two classical non-CNP kernels: the
Bergman kernel on the unit disc $\D$ and the Hardy kernel on the polydisc
$\D^n$ with $n>1$. The latter corresponds to the case
$a_\alpha=1$ and shows that there exists a commuting tuple admitting
two distinct frame vectors that are not similar.

Recall that the Bergman space $L^2_a(\D)$ corresponds to the weight $a_k=\sqrt{k+1}$ for each $k\geq 0$.
\begin{example}
	Let $L$ be the submodule of $L^2_a(\D)$ given in Example~\ref{ex:Bergman}, and let $N$ be a submodule such that $L\neq N$ but $L^\perp$ and $N^\perp$ are similar. Suppose that $A: N^\perp\to L^\perp$ is a module isomorphism. We set $\xi=P_{L^\perp}1$ and $\eta=AP_{N^\perp}1$. Clearly, both  $\xi$ and $\eta$ are weighted frame vectors for the compressed coordinate multiplier $S_z: L^\perp\to L^\perp$.
	
	For each $\varphi\in L^\perp$, since $	\langle\varphi,S_z^k\xi\rangle
	=\langle\varphi,z^k\rangle$,
	we have
	$$
	\Theta_\xi(\varphi)=\sum_{k=0}^\infty\left\langle\varphi,\sqrt{k+1}\,S_z^k\xi	\right\rangle\sqrt{k+1}\,z^k=\varphi,
	$$
	and hence $\Ran\Theta_\xi=L^\perp$. On the other hand,
	\begin{align*}
	\Theta_\eta(\varphi)
	&=
	\sum_{k=0}^\infty
	\left\langle
	\varphi,\sqrt{k+1}\,S_z^k\eta
	\right\rangle
	\sqrt{k+1}\,z^k \\
	&=
	\sum_{k=0}^\infty
	\left\langle
	A^*\varphi,
	\sqrt{k+1}\,S_z^kP_{N^\perp}1
	\right\rangle
	\sqrt{k+1}\,z^k \\
	&=A^*\varphi.
\end{align*} 
Thus $\Ran \Theta_\eta=N^\perp$. Since $L\neq N$, we have $\Ran\Theta_\xi\neq \Ran\Theta_\eta$. Then by Proposition~\ref{prop:equi_frame}, $\xi$ and $\eta$ are not similar.
\end{example}

The Hardy space $H^2(\D^n)$ corresponds to the case $a_\alpha=1$ for all $\alpha\in\Zp^n$.
\begin{theorem}\label{negaans}
	For every $n>1$, there exist a separable Hilbert space $\Kspace$,
	a commuting $n$-tuple $\T=(T_1,\dots,T_n)$ on $\Kspace$, and two frame vectors
	$\xi,\eta\in\Kspace$ for the tuple $\T$ that are not similar.
\end{theorem}
\begin{proof}
	By Theorem~\ref{thm:Hardy}, there exist two distinct but similar
	quotient modules $M^\perp$ and $N^\perp$ of $H^2(\D^n)$.
	Let $A:N^\perp\to M^\perp$ be a module isomorphism, and set $\xi=P_{M^\perp}1$ and $\eta=AP_{N^\perp}1$.
	Then both $\xi$ and $\eta$ are frame vectors for the compressed
	coordinate multipliers $(S_1,\dots,S_n)$ on $M^\perp$.
	By an argument analogous to that in the previous example, we have
	\[
	\operatorname{Ran}\Theta_\xi=M^\perp
	\qquad\text{and}\qquad
	\operatorname{Ran}\Theta_\eta=N^\perp.
	\]
	Since $M^\perp\neq N^\perp$, it follows from
	Proposition~\ref{prop:equi_frame} that $\xi$ and $\eta$ are not
	similar. Thus the conclusion follows by taking $\Kspace=M^\perp$ and $\T=(S_1,\dots,S_n)$.
\end{proof}
We note that this gives a negative answer to the open question mentioned in \cite{BHKLL26}: ``the question of whether the frames generated by every frame representation of any commutative semigroup are equivalent remains open."
More precisely, for each $n>1$, let $\Kspace$, $\T$, $\xi$, and $\eta$ be as in Theorem~\ref{negaans}, and define the semigroup representation
$$
\pi:\Zp^n\to B(\Kspace),\qquad \pi(\alpha)=\T^\alpha.
$$
Then $(\pi,\xi)$ and $(\pi,\eta)$ are frame representations that are not equivalent in the sense of \cite{BHKLL26}.

\vspace{1em}
\smallskip\noindent AI disclosure.
The authors got inspiration from ChatGPT (OpenAI) in developing the proof of Theorem~\ref{thm:Hardy}.

%-------------------------------------reference-----------------------------------------

\end{document}